\documentclass[a4paper,12pt]{article}
\usepackage{amssymb}
\usepackage{amsthm}
\usepackage{tikz}
\usepackage{tikz,pgfplots}
\usetikzlibrary{decorations.markings}
\usepackage{amsmath,amssymb}
\usepackage{todonotes}
\usepackage{url}

\usetikzlibrary{shadows}

\usepackage{color}
\usepackage{graphicx}
\usepackage{placeins,caption} 

\usepackage{hyperref}

\hypersetup{colorlinks=true, linkcolor=blue, citecolor=blue, urlcolor=blue}

\DeclareMathOperator{\mob}{Mob_{gp}}
\DeclareMathOperator{\gp}{gp}

\DeclareMathOperator{\cp}{\,\square\,}

\newcommand{\gpo}{\gp_{\rm o}}

\usepackage[top=3cm,bottom=3cm,left=2.8cm,right=2.8cm]{geometry}
\newtheorem{theorem}{Theorem}[section]
\newtheorem{lemma}[theorem]{Lemma}
\newtheorem{corollary}[theorem]{Corollary}
\newtheorem{proposition}[theorem]{Proposition}

\theoremstyle{definition}

\newcommand{\move}{\rightsquigarrow}

\pgfplotsset{compat=1.18}

\begin{document}
	
	\title{Moving through Cartesian products, coronas and joins in general position}
	\author{Sandi Klav\v{z}ar $^{a,b,c}$ \\ \texttt{\footnotesize sandi.klavzar@fmf.uni-lj.si}  \\
		{\footnotesize ORCID: 0000-0002-1556-4744}
		\and
		Aditi Krishnakumar $^{d}$ \\ \texttt{\footnotesize aditikrishnakumar@gmail.com} \\
		{\footnotesize ORCID: 0000-0002-0954-7542}
		\and
		Dorota Kuziak $^{e}$ \\ \texttt{\footnotesize dorota.kuziak@uca.es} \\ 
		{\footnotesize ORCID: 
			0000-0001-9660-3284}
		\and
		Ethan Shallcross $^{d}$ \\ \texttt{\footnotesize ershallcross@gmail.com} \\ 
		{\footnotesize ORCID: 
			0009-0000-5730-6834}
		\and 
		James Tuite $^{d,f}$ \\ \texttt{\footnotesize james.t.tuite@open.ac.uk} \\ 
		{\footnotesize ORCID: 
			0000-0003-2604-7491}
		\and 
		Ismael G. Yero $^{g}$ \\ \texttt{\footnotesize ismael.gonzalez@uca.es} \\
		{\footnotesize ORCID: 0000-0002-1619-1572}
	}
	
	\maketitle
	
	\begin{center}
		$^a$ Faculty of Mathematics and Physics, University of Ljubljana, Slovenia\\
		
		$^b$ Institute of Mathematics, Physics and Mechanics, Ljubljana, Slovenia\\
		
		$^c$ Faculty of Natural Sciences and Mathematics, University of Maribor, Slovenia\\
		
		$^d$ School of Mathematics and Statistics, Open University, Milton Keynes, UK\\
		
		$^e$ Departamento de Estad\'istica e IO, Universidad de C\'adiz, Algeciras Campus, Spain\\
		
		$^f$ Department of Informatics and Statistics, Klaip\.{e}da University, Lithuania\\
		
		$^g$ Departamento de Matem\'aticas, Universidad de C\'adiz, Algeciras Campus, Spain
	\end{center}
	
	\begin{abstract}
		The general position problem asks for large sets of vertices such that no three vertices of the set lie on a common shortest path. Recently a dynamic version of this problem was defined, called the \emph{mobile general position problem}, in which a collection of robots must visit all the vertices of the graph whilst remaining in general position. In this paper we investigate this problem in the context of Cartesian products, corona products and joins, giving upper and lower bounds for general graphs and exact values for families including grids, cylinders, Hamming graphs and prisms of trees.
	\end{abstract}
	
	
	\noindent
	AMS Subj.\ Class.\ (2020): 05C12, 05C69, 05C76
	
	\section{Introduction}\label{sec:intro}
	
	The \emph{general position problem} originated as a geometric puzzle of Dudeney~\cite{dudeney-1917}, but was first investigated in the context of graph theory in~\cite{ullas-2016} and~\cite{manuel-2018}. A survey of the problem is given in~\cite{survey}. The arXiv version~\cite{manuel-2017arxiv} of the paper~\cite{manuel-2018} gave the following motivation for the problem. Suppose that a collection of robots is stationed on the vertices of a graph. They communicate with each other by sending signals along shortest paths. To avoid their communication being disrupted, we wish that no robot lies on a shortest path between two other robots. Subject to this condition, what is the greatest possible number of robots that we can place on the graph? In fact, the related \emph{mutual-visibility problem} was initially researched in terms of its applications in robotic navigation and communication (see~\cite{aljohani-2018a,aljohani-2018b,diluna-2017} for a partial overview) and was only recently considered in a pure mathematics context~\cite{DiStef}.
	
	However, this picture lacks an important feature of real world robotic navigation: the general position problem is `static', whereas in applications the robots will typically need to move around the network. Watching the mobile delivery robots created by Starship Technologies\textregistered\ \cite{starship} inspired the authors of~\cite{klavzar-2023} to consider a dynamic version of the general position problem, in which robots move through the vertices of a graph whilst remaining in general position, and such that every vertex is visited by a robot at least once (it is assumed that one robot moves to an adjacent vertex at each step and no vertex can contain more than one robot). The paper~\cite{klavzar-2023} considered this problem for block graphs, rooted products, Kneser graphs, unicyclic graphs, complete multipartite graphs and line graphs of complete graphs. A mobile version of the closely related mutual-visibility problem was recently treated in~\cite{Dettlaf2024}. In addition, the paper~\cite{Boshar2024} considers the version of the mobile general position problem with the stricter condition that every vertex must be visited by every robot.
	
	The general position problem has been investigated for a wide variety of graphs, but there is a particularly extensive literature on general position sets in Cartesian products, see~\cite{klavzar-2021,KlavzarRus,KorzeVesel,tian-2021a,tian-2021b}. General position sets in other graph products were discussed in~\cite{ghorbani-2021}. Cartesian products have also been considered in the setting of variants of the general position number, such as the lower general position number~\cite{Welton}, the mutual-visibility number~\cite{Cicerone2023,DiStef}, the monophonic position number~\cite{mpproduct}, the lower mutual-visibility number~\cite{lowermv}, general position polynomials~\cite{polynomial}, edge general position numbers~\cite{edgeposition}, total mutual-visibility~\cite{Kuziak2023}, the variety of general position problems~\cite{tian-2024+} and general position and mutual-visibility colourings~\cite{ChaDiSreeThoTui2024+,klavzar-2025+}. In this paper, we examine the mobile general position problem in Cartesian products, together with the coronas and joins of graphs.
	
	The plan of the paper is as follows. In Subsection~\ref{subsec-def} we introduce the formal definitions of the main concepts that we shall use in our exposition. In Section~\ref{sec:Cartesian} we give bounds for the mobile general position number of Cartesian products and discuss their sharpness. In Section~\ref{sec:prisms} we determine this number for Cartesian products involving paths, including grids, some cylinders and prisms of trees. Section~\ref{sec:corona+joins} discusses the mobile general position problem for corona products and joins of graphs. We conclude with some open problems in Section~\ref{sec:conclude}.
	
	\subsection{Formal definitions and preliminaries} \label{subsec-def}
	
	We will write $[n]$ for $\{ 1,\dots, n\} $ and $[m,n]$ for $\{ m,m+1,\dots ,n-1,n\} $ when $m \leq n$. A \emph{graph} $G = (V(G),E(G))$ consists of a set $V(G)$ of vertices that are connected by a set of edges $E(G)$. All graphs that we consider are simple and undirected. We will write $u \sim v$ to indicate that $u$ and $v$ are adjacent in $G$, and will denote the set of neighbours of $u$ by $N_G(u)$, or simply by $N(u)$ if the graph is clear from the context. The \emph{degree} $\deg (u)$ of a vertex $u$ is $|N(u)|$. A vertex of degree one is a \emph{leaf}, and the number of leaves in a graph $G$ will be denoted by $\ell (G)$. A complete graph with $n$ vertices will be written as $K_n$ and a complete bipartite graph with partite sets of size $n$ and $m$ as $K_{n,m}$. 
	
	A \emph{path} $P_{r}$ in $G$ is a sequence $u_1,u_2,\dots ,u_r$ of distinct vertices such that $u_i \sim u_{i+1}$ for $i\in [r-1]$ and the \emph{length} of this path is $r-1$. A \emph{cycle} $C_r$ is a sequence $u_1,u_2,\dots ,u_r$ such that $u_i \sim u_{i+1}$ for $i \in [r-1]$ and also $u_1 \sim u_r$. The \emph{distance} between vertices $u,v \in V(G)$ is the length of a shortest $u,v$-path in $G$. We will typically identify the vertices of a cycle $C_n$ with $\mathbb{Z}_n$ and the vertices of a path $P_n$ with $[n]$ in the natural manner. A subgraph $X$ of $G$ is \emph{convex} if for any $u,v \in V(X)$ we have $V(P) \subseteq V(X)$ for any shortest $u,v$-path $P$ in $G$. The subgraph $G[X]$ induced by a subset $X \subseteq V(G)$ is the subgraph with vertex set $X$ such that $u,v \in X$ are adjacent in $G[X]$ if and only if they are adjacent in $G$.
	
	If $S \subseteq V(G)$ has the property that no three vertices of $S$ lie on a common shortest path of $G$, then we say that $S$ is a \emph{general position set} of $G$. The largest possible number of vertices in a general position set of $G$ is the \emph{general position number} of $G$, denoted by $\gp (G)$. Any largest general position set is referred to as a \emph{gp-set}. Fig.~\ref{fig:Petersen} shows two general position sets in the well-known Petersen graph. The red vertices on the left form a maximal but non-optimal general position set of order four, whilst the six red vertices on the right constitute a $\gp $-set of the Petersen graph, and so the Petersen graph has general position number six. Note that in both cases no shortest path between any pair of red vertices passes through a third red vertex.
	
	\begin{figure}[ht!]
		\centering
		\begin{tikzpicture}[scale=0.4, vertex_style/.style={circle, ball color=black},vertex_style_1/.style={circle, ball color=red},
			edge_style/.style={thick, black,drop shadow={opacity=0.4}}]
			
			\begin{scope}[rotate=90]
				\foreach \x/\y in {216/4,288/5}{
					\node[vertex_style_1] (\y) at (canvas polar cs: radius=2.5cm,angle=\x){};}
				\foreach \x/\y in {0/1,72/2,144/3}{
					\node[vertex_style] (\y) at (canvas polar cs: radius=2.5cm,angle=\x){};}
				\foreach \x/\y in {0/6,144/8}{
					\node[vertex_style_1] (\y) at (canvas polar cs: radius=5cm,angle=\x){};}
				\foreach \x/\y in {72/7,216/9,288/10}{
					\node[vertex_style] (\y) at (canvas polar cs: radius=5cm,angle=\x){};}
			\end{scope}
			
			\foreach \x/\y in {1/6,2/7,3/8,4/9,5/10}{
				\draw[edge_style] (\x) -- (\y);}
			
			\foreach \x/\y in {1/3,2/4,3/5,4/1,5/2}{
				\draw[edge_style] (\x) -- (\y);}
			
			\foreach \x/\y in {6/7,7/8,8/9,9/10,10/6}{
				\draw[edge_style] (\x) -- (\y);}
		\end{tikzpicture}
		\hspace*{1.5cm}
		\begin{tikzpicture}[scale=0.4, vertex_style/.style={circle, ball color=black},vertex_style_1/.style={circle, ball color=red},
			edge_style/.style={thick, black,drop shadow={opacity=0.4}}]
			
			\begin{scope}[rotate=90]
				\foreach \x/\y in {0/1,216/4,288/5}{
					\node[vertex_style_1] (\y) at (canvas polar cs: radius=2.5cm,angle=\x){};}
				\foreach \x/\y in {72/2,144/3}{
					\node[vertex_style] (\y) at (canvas polar cs: radius=2.5cm,angle=\x){};}
				\foreach \x/\y in {72/7,144/8,288/10}{
					\node[vertex_style_1] (\y) at (canvas polar cs: radius=5cm,angle=\x){};}
				\foreach \x/\y in {0/6,216/9}{
					\node[vertex_style] (\y) at (canvas polar cs: radius=5cm,angle=\x){};}
			\end{scope}
			
			\foreach \x/\y in {1/6,2/7,3/8,4/9,5/10}{
				\draw[edge_style] (\x) -- (\y);}
			
			\foreach \x/\y in {1/3,2/4,3/5,4/1,5/2}{
				\draw[edge_style] (\x) -- (\y);}
			
			\foreach \x/\y in {6/7,7/8,8/9,9/10,10/6}{
				\draw[edge_style] (\x) -- (\y);}
		\end{tikzpicture}
		\caption{A (non-optimal) general position set (left) and a gp-set (right) of the Petersen graph are shown in red.}\label{fig:Petersen}
	\end{figure}
	
	If a robot is located at a vertex $u$ and $u \sim v$, then we indicate the movement of the robot from $u$ to $v$ along the edge $uv$ by $u \move v$ and refer to this as a \emph{move}. Suppose that we assign exactly one robot to each vertex of a general position set $S$. If a robot is stationed at a vertex $u$ of $S$, then the move $u \move v$ is called a \emph{legal move} if i) $v \notin S$ (thereby avoiding having more than one robot per vertex at any stage) and ii) the new set $(S\setminus\{ u\} ) \cup \{ v\}$ is also a general position set. A configuration of robots on a general position set of $G$ is called a \emph{mobile general position set} if there is a sequence of legal moves starting from $S$ such that every vertex of $G$ is visited at least once by some robot. The \emph{mobile general position number}, written $\mob (G)$, is the largest number of robots in a mobile general position set of $G$. We will refer to a largest possible configuration of robots in mobile general position as a {\em $\mob$-set}. 
	
	Fig.~\ref{fig:Petersen-moves} shows a few configurations of a $\mob$-set for the Petersen graph. The red arrows indicate the legal moves that lead from one configuration to the next. Notice that the last configuration is equivalent to the first up to symmetry. By repeating (symmetrical equivalents of) this sequence of legal moves every vertex of the Petersen graph can be visited. It is shown in~\cite{klavzar-2023} that this mobile general position has the largest possible cardinality amongst all such sets, i.e. it is a $\mob$-set.
	
	\begin{figure}[ht!]
		\centering
		\begin{tikzpicture}[scale=0.3, vertex_style/.style={circle, ball color=black},vertex_style_1/.style={circle, ball color=red},
			edge_style/.style={thick, black,drop shadow={opacity=0.4}}]
			
			\begin{scope}[rotate=90]
				\foreach \x/\y in {216/4,288/5}{
					\node[vertex_style_1] (\y) at (canvas polar cs: radius=2.5cm,angle=\x){};}
				\foreach \x/\y in {0/1,72/2,144/3}{
					\node[vertex_style] (\y) at (canvas polar cs: radius=2.5cm,angle=\x){};}
				\foreach \x/\y in {0/6,144/8}{
					\node[vertex_style_1] (\y) at (canvas polar cs: radius=5cm,angle=\x){};}
				\foreach \x/\y in {72/7,216/9,288/10}{
					\node[vertex_style] (\y) at (canvas polar cs: radius=5cm,angle=\x){};}
			\end{scope}
			
			\foreach \x/\y in {1/6,2/7,3/8,4/9,5/10}{
				\draw[edge_style] (\x) -- (\y);}
			
			\foreach \x/\y in {1/3,2/4,3/5,4/1,5/2}{
				\draw[edge_style] (\x) -- (\y);}
			
			\foreach \x/\y in {6/7,7/8,8/9,9/10,10/6}{
				\draw[edge_style] (\x) -- (\y);}
			
			\draw[ultra thick, color=red,->] (6) -- (7);
			
		\end{tikzpicture}
		\hspace*{0.3cm}
		\begin{tikzpicture}[scale=0.3, vertex_style/.style={circle, ball color=black},vertex_style_1/.style={circle, ball color=red},
			edge_style/.style={thick, black,drop shadow={opacity=0.4}}]
			
			\begin{scope}[rotate=90]
				\foreach \x/\y in {216/4,288/5}{
					\node[vertex_style_1] (\y) at (canvas polar cs: radius=2.5cm,angle=\x){};}
				\foreach \x/\y in {0/1,72/2,144/3}{
					\node[vertex_style] (\y) at (canvas polar cs: radius=2.5cm,angle=\x){};}
				\foreach \x/\y in {72/7,144/8}{
					\node[vertex_style_1] (\y) at (canvas polar cs: radius=5cm,angle=\x){};}
				\foreach \x/\y in {0/6,216/9,288/10}{
					\node[vertex_style] (\y) at (canvas polar cs: radius=5cm,angle=\x){};}
			\end{scope}
			
			\foreach \x/\y in {1/6,2/7,3/8,4/9,5/10}{
				\draw[edge_style] (\x) -- (\y);}
			
			\foreach \x/\y in {1/3,2/4,3/5,4/1,5/2}{
				\draw[edge_style] (\x) -- (\y);}
			
			\foreach \x/\y in {6/7,7/8,8/9,9/10,10/6}{
				\draw[edge_style] (\x) -- (\y);}
			
			\draw[ultra thick, color=red,->] (4) -- (1);
		\end{tikzpicture}
		\hspace*{0.3cm}
		\begin{tikzpicture}[scale=0.3, vertex_style/.style={circle, ball color=black},vertex_style_1/.style={circle, ball color=red},
			edge_style/.style={thick, black,drop shadow={opacity=0.4}}]
			
			\begin{scope}[rotate=90]
				\foreach \x/\y in {0/1,288/5}{
					\node[vertex_style_1] (\y) at (canvas polar cs: radius=2.5cm,angle=\x){};}
				\foreach \x/\y in {72/2,144/3,216/4}{
					\node[vertex_style] (\y) at (canvas polar cs: radius=2.5cm,angle=\x){};}
				\foreach \x/\y in {72/7,144/8}{
					\node[vertex_style_1] (\y) at (canvas polar cs: radius=5cm,angle=\x){};}
				\foreach \x/\y in {0/6,216/9,288/10}{
					\node[vertex_style] (\y) at (canvas polar cs: radius=5cm,angle=\x){};}
			\end{scope}
			
			\foreach \x/\y in {1/6,2/7,3/8,4/9,5/10}{
				\draw[edge_style] (\x) -- (\y);}
			
			\foreach \x/\y in {1/3,2/4,3/5,4/1,5/2}{
				\draw[edge_style] (\x) -- (\y);}
			
			\foreach \x/\y in {6/7,7/8,8/9,9/10,10/6}{
				\draw[edge_style] (\x) -- (\y);}
			
			\draw[ultra thick, color=red,->] (8) -- (9);
		\end{tikzpicture}
		\hspace*{0.3cm}
		\begin{tikzpicture}[scale=0.3, vertex_style/.style={circle, ball color=black},vertex_style_1/.style={circle, ball color=red},
			edge_style/.style={thick, black,drop shadow={opacity=0.4}}]
			
			\begin{scope}[rotate=90]
				\foreach \x/\y in {0/1,288/5}{
					\node[vertex_style_1] (\y) at (canvas polar cs: radius=2.5cm,angle=\x){};}
				\foreach \x/\y in {72/2,144/3,216/4}{
					\node[vertex_style] (\y) at (canvas polar cs: radius=2.5cm,angle=\x){};}
				\foreach \x/\y in {72/7,216/9}{
					\node[vertex_style_1] (\y) at (canvas polar cs: radius=5cm,angle=\x){};}
				\foreach \x/\y in {0/6,144/8,288/10}{
					\node[vertex_style] (\y) at (canvas polar cs: radius=5cm,angle=\x){};}
			\end{scope}
			
			\foreach \x/\y in {1/6,2/7,3/8,4/9,5/10}{
				\draw[edge_style] (\x) -- (\y);}
			
			\foreach \x/\y in {1/3,2/4,3/5,4/1,5/2}{
				\draw[edge_style] (\x) -- (\y);}
			
			\foreach \x/\y in {6/7,7/8,8/9,9/10,10/6}{
				\draw[edge_style] (\x) -- (\y);}
		\end{tikzpicture}
		
		\caption{A sequence of legal moves (shown by red arrows) for a $\mob$-set of the Petersen graph.}\label{fig:Petersen-moves}
	\end{figure}
	
	In~\cite{tian-2024+} the variety of general position problems in graphs was introduced, including general position sets, outer general position sets, dual general position sets, and total general position sets. For our purposes we recall the following definitions. If $X\subseteq V(G)$, then $u,v\in V(G)$ are {\em $X$-positionable} if for any shortest $u,v$-path $P$ we have $V(P)\cap X \subseteq \{u,v\}$. Hence $X$ is a general position set if all pairs $u,v\in X$ are $X$-positionable. If it also holds that every pair $u,v$ with $u\in X$ and $v\in V(G) \setminus X$ is $X$-positionable, then $X$ is an \emph{outer general position set}. The cardinality of a largest outer general position set of $G$ is denoted by $\gpo(G)$ and is called the {\em outer general position number}. It is shown in~\cite{tian-2024+} that outer general position sets coincide with sets of mutually maximally distant vertices. (The latter concept was introduced in~\cite{Oellermann-2007}, see also the related survey~\cite{Kuziak-2018}.) In particular, in a block graph the outer general position sets are the sets of simplicial vertices; in the case of a tree this yields $\gp (T) = \gpo (T) = \ell (T)$.
	
	%%%%%%%%%%%%%%%%%%%%%%%%%%%%%%%%%%%%%%%%%%%%%
	\section{Bounds for Cartesian products}
	\label{sec:Cartesian}
	%%%%%%%%%%%%%%%%%%%%%%%%%%%%%%%%%%%%%%%%%%%%%
	
	Recall that the {\em Cartesian product} $G\cp H$ of graphs $G$ and $H$ satisfies $V(G\cp H) = V(G)\times V(H)$ and $(g,h)(g',h')\in E(G\cp H)$ if either $gg'\in E(G)$ and $h=h'$, or $g=g'$ and $hh'\in E(H)$. A \emph{$G$-layer} is a subgraph of $G \cp H$ induced by $V(G) \times \{ h\}$ for some $h \in V(H)$, which will be denoted by $G^h$, with a similar definition for $H$-layers $^gH$, where $g\in V(G)$. Likewise, if $P$ is a path $u_1,\dots ,u_{r}$ in $G$ and $h \in V(H)$, then we will denote the path $(u_1,h),(u_2,h),\dots ,(u_{r},h)$ in $G \cp H$ by $P^h$ (with an analogous definition of $^gQ$ for a path $Q$ in $H$ and $g \in V(G)$).
	
	We begin by deriving some bounds on $\mob (G \cp H)$. A trivial upper bound is $\mob (G \cp H) \leq \gp (G \cp H)$. Proposition~\ref{prop:bounds-for-cp} gives two lower bounds in terms of the mobile and outer general position numbers of the factors.

	\begin{proposition}
		\label{prop:bounds-for-cp}
		% \label{lem:mid position bound}
		For any connected graphs $G$ and $H$ of order at least two, the following hold. 
		\begin{enumerate}
			\item[(i)] $\mob(G\cp H) \ge \max\{\mob(G), \mob(H)\}$.
			\item[(ii)] $\mob (G \cp H) \geq \max \{ \gpo(G), \gpo(H)\}$. 
		\end{enumerate}
	\end{proposition}
	
	\begin{proof}
		$(i)$ Let $S$ be a $\mob$-set of $G$ and let $h\in V(H)$. We first position the robots at the vertices of the set $S \times \{ h\}$. As $G^h$ is a convex subgraph of $G \cp H$, robots initially stationed at the vertices of $S \times \{ h\} $ can visit every vertex of $G^h$ by a sequence of legal moves, all the time remaining inside $G^h$ and in general position in $G \cp H$. Now, whenever a robot visits a vertex $(g,h)$ in this layer, this robot can visit all the vertices in the $H$-layer $^gH$ corresponding to $g$ by a sequence of legal moves and then return to $(g,h)$. Hence $\mob(G\cp H) \ge \mob(G)$. By symmetry, $\mob(G\cp H) \ge \mob(H)$. 
		
		$(ii)$ Let $S = \{ u_1,\dots ,u_r\} $ be an outer general position set of $G$ of cardinality $\gpo(G)$ and start with robots positioned at each vertex of $S \times \{ h\} $ for some $h \in V(H)$. Let $R_i$ be the robot at $(u_i, h)$ for $i \in [r]$. Also, for each $i\in [r]$, let $G_i$ be the connected component containing $u_i$ in $G \setminus (S\setminus\{ u_i\} )$. 
		
		Let $h' \in V(H)\setminus\{ h\}$ and let $Q$ be a shortest $h,h'$-path in $H$. The robot $R_i$ can follow the path ${}^{u_i}Q$ from $(u_i,h)$ to reach the vertex $(u_i,h')$ by legal moves. At this point, $R_i$ can visit all the vertices of $V(G_i) \times \{ h' \}$ by a sequence of legal moves. To see this, notice that the shortest paths between the robots remaining in $G^h$ lie within the layer $G^h$, whilst the shortest paths from $R_i$ to any $R_j$, $i \neq j$, do not pass through a third robot by the outer general position property. Afterwards $R_i$ can return to $(u_i,h)$ by performing these legal moves in the reverse order. As this holds for any of the robots and any $h' \in V(H) \setminus \{ h\} $, this allows us to perform a sequence of legal moves so that any vertex of $G \cp H$ outside $G^h$ is visited, since $\bigcup _{i=1}^{r}V(G_i) = V(G)$. Having done this, we return all the robots to their original positions in $S \times \{ h\} $ using legal moves.
		
		Finally, let $h^{\prime } \in N_H(h)$. For $i\in [r]$ we move the robot $R_i$ from $(u_i,h)$ to $(u_i,h^{\prime })$ in sequence; as $S$ is in general position each of these moves is legal. Afterwards the robots will occupy the set $S \times \{ h^{\prime }\}$. The previous reasoning applied to $G^{h^{\prime }}$ now shows that the robots can visit each vertex of $V(G) \times \{ h\} $ by legal moves. Thus $\mob (G \cp H) \geq \gpo(G)$ and, by a symmetric argument, $\mob (G \cp H) \geq \gpo(H)$.
	\end{proof}
	
	We now show that both lower bounds in Proposition~\ref{prop:bounds-for-cp} are sharp by considering the Cartesian products $K_r \cp P_s$. We recall that~\cite[Theorem 3.2]{tian-2021a} implies that $\gp(K_r \cp P_s) = r+1$ for $s\ge 3$. Observe that for $r,s \geq 2$ it holds that $\max \{ \mob{(K_r)}, \mob{(P_s)} \} = \max\{ \gpo{(K_r)}, \gpo{(P_s)} \} = r$.
	
	\begin{proposition}\label{prop:clique cp path}
		For all positive integers $r,s \geq 2$, \[\mob(K_r \cp P_s) = r.\]
	\end{proposition}
	
	\begin{proof}
		Let $r,s \geq 2$. By Proposition~\ref{prop:bounds-for-cp} we have $\mob(K_r \cp P_s) \geq r$. We now show that $\mob(K_r \cp P_s) \leq r$. Let $V(P_s) = [s]$ and suppose for a contradiction that there exists a mobile general position set $S$ of $K_r \cp P_s$ with $|S| > r$. Then choose $x \in V(K_r)$, such that $(x, i), (x,j) \in V(K_r \cp P_s)$ are occupied by robots $R_1, R_2$, where $i < j$. 
		
		First, notice that for any $y \in V(K_r)$ and $j \leq k \leq s$ the vertex $(y,k)$ cannot be occupied by a robot in the initial configuration $S$, for otherwise $(x,j)$ would lie on a shortest path between $(x,i)$ and $(y,k)$. Similarly, every other vertex $(y,k)$ with $k \leq i$ is unoccupied. The same reasoning shows that any other $P_s$-layer can contain at most one robot, since if there are robots at $(y,k)$ and $(y,k')$, where $y \in V(K_r) \setminus \{ x\} $ and $i < k < k' < j$, then $(y,k')$ would lie on a shortest $(y,k),(x,j)$-path. Thus, we must have $|S| = r+1$ and each layer $^yP_s$ contains exactly one robot for $y \neq x$. Moreover, each of the remaining $r-1$ robots different from $R_1$ and $R_2$ are located at vertices $(y,k)$ such that $i<k<j$.
		
		It now follows that neither $R_1$ nor $R_2$ can cross to another $P_s$-layer by a sequence of legal moves. Otherwise, suppose that robots $R_1$ and $R_2$ are stationed at $(x,i')$ and $(x,j')$ respectively just before $R_1$ moves to a different $P_s$-layer by the legal move $(x,i') \move (y,i')$, $y \in V(K_r) \setminus \{ x\} $. Since there is a robot at some vertex $(y,k')$ with $i' < k' < j'$, after this move the robot at $(y,k')$ would lie on a shortest path from $(y,i')$ to $(x,j')$. As a result, no robot can visit any vertex in $(V(K_n)\setminus \{ x\} ) \times \{ 1,s\} $, a contradiction. Thus at most $r$ robots can traverse $K_r \cp P_s$ in general position.
	\end{proof}
	
	To see that the lower bounds of Proposition~\ref{prop:bounds-for-cp} are independent in general, consider the following examples. If $n\ge 2$, then $\mob(K_{n,n}) = 2$ and $\gpo(K_{n,n}) = n$. Hence the bound $(i)$ yields $\mob(K_{n,n} \cp K_{n,n})\ge 2$, whilst $(ii)$ yields $\mob(K_{n,n} \cp K_{n,n}) \ge n$. On the other hand, if $n \ge 7$, then $\mob(C_n) = 3$ and $\gpo(C_n) = 2$, hence the bound $(i)$ is better for $\mob(C_{n} \cp C_{n})$ if $n \geq 7$. Moreover, just after Theorem~\ref{thm:Hamming-graphs} we will demonstrate that the mobile general position number of a graph can be arbitrarily larger than its outer general position number, so that bound $(i)$ can also be arbitrarily larger than bound $(ii)$.
	
	From~\cite[Theorem 3.2]{ghorbani-2021} we recall that if $n\ge 2$ and $m\ge 2$, then $\gp(K_n\cp K_m) = n + m - 2$. We now sharpen this result by demonstrating that a gp-set of $K_n\cp K_m$ as constructed in~\cite[Theorem 3.2]{ghorbani-2021} is essentially unique as soon as $n\ge 3$ and $m\ge 3$. 
	
	\begin{lemma}
		\label{lem:unique-gp-in-Hamming}
		Let $n\ge 3$, $m\ge 3$, $V(K_n) = [n]$, $V(K_m) = [m]$, and let $X$ be a gp-set of $K_n\cp K_m$. Then there exist $i\in [n]$ and $j\in [m]$ such that 
		$$X = \left( V((K_n)^j) \cup V(^i(K_m))\right )\setminus \{(i,j)\}\,.$$ 
	\end{lemma}
	
	\begin{proof}
		Let $X$ be a gp-set of $K_n\cp K_m$. Then, as stated above, $|X| = n + m - 2$. Note that the set $X$ contains at most two vertices from every induced copy of $C_4$ of $K_n\cp K_m$. This fact implies the following: 
		
		\medskip\noindent
		{\bf Claim A}: if $(i,j), (i,j')\in X$, where $j\ne j'$, then 
		$X \cap \left( V((K_n)^j) \cup V((K_n)^{j'})  \right) = \{(i,j), (i,j')\}$. 
		
		\medskip
		If $X$ contains all the vertices of some $K_n$-layer, then by Claim~A, $X$ contains no other vertices, implying that $|X| = n < n + m - 2$, which is not possible. Let $j\in [m]$ be such that $t = |X \cap V((K_n)^{j})|$ is as large as possible. Note that $t\ge 2$, for otherwise we would have $|X|\le m$. Moreover, by the above, $t \le n-1$. If $t = n-1$ and $i\in [n]$ is the index for which $(i,j)\notin X$, then using Claim~A again, we have that $X \subseteq V((K_n)^j) \cup V(^i(K_m))$. Since $|X| = n+m-2$, we conclude that $X$ has the required structure in this case. Suppose finally that $t=n-k\ge 2$, where $k\ge 2$. We may assume without loss of generality that $X \cap V((K_n)^j ) = \{(1,j), \ldots, (n-k,j)\}$. Then, using Claim~A once more, 
		$$X \cap \bigcup_{\ell = 1}^{n-k} V((K_n)^{\ell}) = \{(1,j), \ldots, (n-k,j)\}\,.$$
		Notice that the subgraph $H$ of $K_n\cp K_m$ induced by the vertex set 
		$$\{n-k+1, \ldots, n\} \times \left( \{1, \ldots, j-1\} \cup \{j+1, \ldots, m \} \right)$$
		is isomorphic to $K_{k} \cp K_{m-1}$. Since it is a convex subgraph of $K_n\cp K_m$, the intersection $X\cap V(H)$ is a general position set of $H$. Therefore, 
		$$|X \cap V(H)| \le \gp(K_{k} \cp K_{m-1}) = k + (m-1) - 2 = k + m - 3\,.$$ 
		By our assumption on $t$ we have $X\cap \{(n-k+1,j), \ldots, (n,j)\} = \emptyset$. Thus we can conclude that $|X| \le (n - k) + (k + m - 3) = n + m -3$, a contradiction. 
		
		We have thus proved that $|X| = n + m - 2$ holds only in the case when $t = n-1$ and $X$ has the structure as claimed. 
	\end{proof}
	
	Lemma~\ref{lem:unique-gp-in-Hamming} is illustrated in Fig.~\ref{fig:K_7-byK_5}, where the gp-set of $K_7\cp K_5$ corresponding to the vertex $(i,j)$ is shown. We now use this result to find the mobile general position number of $K_n \cp K_m$.
	
	\begin{figure}[ht!]
		\begin{center}
			\begin{tikzpicture}[scale=0.7,style=thick]
				\tikzstyle{every node}=[draw=none,fill=none]
				\def\vr{3pt} % \vr = vertex radius;  Set \vr = 2/scale for uniform sizing of vertices
				
				\begin{scope}[yshift = 0cm, xshift = 0cm]
					%% vertices defined %%
					\path (0,0) coordinate (v1);
					\path (1,0) coordinate (v2);
					\path (2,0) coordinate (v3);
					\path (3,0) coordinate (v4);
					\path (4,0) coordinate (v5);
					\path (5,0) coordinate (v6);
					\path (6,0) coordinate (v7);
					%% edges %%
					%% vertices %%%
					\foreach \i in {1,...,7}
					{
						\draw (v\i)  [fill=white] circle (\vr);
					}
					\draw (v4)  [fill=black] circle (\vr);
					%% text %%
					\draw (3,-1) node {$i$};
					\draw (-1,2) node {$j$};
					
				\end{scope}
				
				\begin{scope}[yshift = 1cm, xshift = 0cm]
					%% vertices defined %%
					\path (0,0) coordinate (v1);
					\path (1,0) coordinate (v2);
					\path (2,0) coordinate (v3);
					\path (3,0) coordinate (v4);
					\path (4,0) coordinate (v5);
					\path (5,0) coordinate (v6);
					\path (6,0) coordinate (v7);
					%% vertices %%%
					\foreach \i in {1,...,7}
					{
						\draw (v\i)  [fill=white] circle (\vr);
					}
					\draw (v4)  [fill=black] circle (\vr);
				\end{scope}
				
				\begin{scope}[yshift = 2cm, xshift = 0cm]
					%% vertices defined %%
					\path (0,0) coordinate (v1);
					\path (1,0) coordinate (v2);
					\path (2,0) coordinate (v3);
					\path (3,0) coordinate (v4);
					\path (4,0) coordinate (v5);
					\path (5,0) coordinate (v6);
					\path (6,0) coordinate (v7);
					%% vertices %%%
					\foreach \i in {1,...,7}
					{
						\draw (v\i)  [fill=white] circle (\vr);
					}
					\draw (v1)  [fill=black] circle (\vr);
					\draw (v2)  [fill=black] circle (\vr);
					\draw (v3)  [fill=black] circle (\vr);
					\draw (v5)  [fill=black] circle (\vr);
					\draw (v6)  [fill=black] circle (\vr);
					\draw (v7)  [fill=black] circle (\vr);
				\end{scope}
				
				\begin{scope}[yshift = 3cm, xshift = 0cm]
					%% vertices defined %%
					\path (0,0) coordinate (v1);
					\path (1,0) coordinate (v2);
					\path (2,0) coordinate (v3);
					\path (3,0) coordinate (v4);
					\path (4,0) coordinate (v5);
					\path (5,0) coordinate (v6);
					\path (6,0) coordinate (v7);
					%% vertices %%%
					\foreach \i in {1,...,7}
					{
						\draw (v\i)  [fill=white] circle (\vr);
					}
					\draw (v4)  [fill=black] circle (\vr);
				\end{scope}
				
				\begin{scope}[yshift = 4cm, xshift = 0cm]
					%% vertices defined %%
					\path (0,0) coordinate (v1);
					\path (1,0) coordinate (v2);
					\path (2,0) coordinate (v3);
					\path (3,0) coordinate (v4);
					\path (4,0) coordinate (v5);
					\path (5,0) coordinate (v6);
					\path (6,0) coordinate (v7);
					%% vertices %%%
					\foreach \i in {1,...,7}
					{
						\draw (v\i)  [fill=white] circle (\vr);
					}
					\draw (v4)  [fill=black] circle (\vr);
				\end{scope}

			\end{tikzpicture}
		\end{center}
		\caption{Canonical gp-set in $K_7\cp K_5$}
		\label{fig:K_7-byK_5}
	\end{figure}

	\begin{theorem}\label{thm:Hamming-graphs}
		If $n\ge m\ge 1$, then
		$$\mob(K_n\cp K_m) = 
		\left\{
		\begin{array}{ll}
			n;     &  m\in [2], \\
			n+m-3;     &  m\ge 3.
		\end{array} \right.
		$$
	\end{theorem}
	
	\begin{proof}
		If $m = 1$, then $K_n\cp K_1 \cong K_n$, thus $\gp(K_n\cp K_1) = \mob(K_n\cp K_1) = n$. The case $m = 2$ follows from Proposition~\ref{prop:clique cp path}.
		
		Let $m\ge 3$. By Lemma~\ref{lem:unique-gp-in-Hamming}, every gp-set of $K_n\cp K_m$ has the canonical form as illustrated in Fig.~\ref{fig:K_7-byK_5}. It is straightforward to check that no robot placed in such a set can make a legal move. This implies that $\mob(K_n \cp K_m) < \gp(K_n\cp K_m) = n + m - 2$. 
		
		To complete the proof, we are going to show that there exists a mobile general position set of cardinality $n + m - 3$. Station $n+m-3$ robots on the vertices of $S = \{(2,1), (3,1), \ldots, (n,1)\} \cup \{(1,3), (1,4), \ldots, (1,m)\}$. This set is a general position set, since it is a subset of the canonical gp-set of $K_n\cp K_m$, as illustrated in Fig.~\ref{fig:K_7-byK_5}. Then we perform the following sequence of moves:  
		\begin{itemize}
			\item $(2,1) \move (2,2)$, $(3,1) \move (3,2)$, $\ldots$, $(n,1) \move (n,2)$, 
			\item $(1,3) \move (1,1)$.
		\end{itemize}
		Observe that each of these moves is legal. Thus, the new set occupied by the robots is a general position set of $K_n\cp K_m$. Next, this process of legal moves can be repeated $m-2$ times, i.e.\ for $2 \leq j \leq m-1$ perform the sequence \[ (2,j) \move (2,j+1),(3,j) \move (3,j+1),\ldots ,(n,j)\move (n,j+1),(1,j+2) \move (1,j-1) ,\] (skipping the final undefined move). In this way, the robots initially positioned at $(2,1), (3,1), \ldots, (n,1)$ will visit the vertices of the set $\{2, 3, \ldots, n\} \times [m]$ by legal moves, while the remaining vertices can be visited by legal moves by the robots initially positioned at $(1,3), (1,4), \ldots, (1,m)$. This concludes our argument for the existence of a mobile general position set of cardinality $n + m - 3$.
	\end{proof}
	
	By Theorem~\ref{thm:Hamming-graphs} we have $\mob (K_n \cp K_n) = 2n-3$, whilst it follows from~\cite[Corollary 4.4]{tian-2024+} that $\gpo(K_n\cp K_n) = n$, so that the mobile general position number of a graph can be arbitrarily larger than the outer general position number. Theorem~\ref{thm:Hamming-graphs} also demonstrates that the mobile general position number of a Cartesian product can be arbitrarily larger than both bounds in Proposition~\ref{prop:bounds-for-cp}.
	
	Finally, we give an example (Cartesian products of stars) that shows that the mobile general position number of a non-trivial Cartesian product can be arbitrarily smaller than its general position number. 
	
	\begin{proposition}
		For all $r \geq 1$, there exist graphs $G, H$ for which \[ \gp(G \cp H) - \mob(G \cp H) = r.\]
	\end{proposition}
	
	\begin{proof}
		For $k \geq 2$, let $V(K_{1,k}) = \{0\} \cup [k]$, with $0$ being the vertex of degree $k$. It follows from~\cite[Theorem 1]{tian-2021b} that $\gp (K_{1,k} \cp K_{1,k}) = 2k$. We will show that $\mob(K_{1,k} \cp K_{1,k}) = k+1$, so that $\gp(K_{1,k} \cp K_{1,k}) - \mob(K_{1,k} \cp K_{1,k}) = k-1$ and the result follows on setting $k = r+1$.
		
		Let $S$ be any mobile general position set of $K_{1,k} \cp K_{1,k}$. Any pair of adjacent vertices in the Cartesian product of trees is a maximal general position set~\cite{Welton}, so if $|S| > 2$ the robots must always occupy an independent set. We can start at the stage that a robot is stationed at $(0,0)$. If there are robots at vertices $(u_1,v_1)$ and $(u_2,v_2)$ with $u_1,v_1,u_2,v_2 \in [k]$, $u_1 \neq u_2$ and $v_1 \neq v_2$, then there would be a shortest $(u_1,v_1),(u_2,v_2)$-path through $(0,0)$; hence all the remaining robots lie in a set $\{ i\} \times [k]$ or $[k] \times \{ j\}$ for some $i,j \in [k]$. In either case, we have $|S| \leq k+1$.
		
		Finally, we show that $\mob(K_{1,k} \cp K_{1,k}) \geq k+1$. Consider the set $S = \{ (0,0) \} \cup ([k] \times \{ 1\} )$. First, move $(0,0) \move (0,2)$, followed by $(i,1) \move (i,0)$ for $i \in [2, k]$ and then $(0,2) \move (1,2)$. By relabelling as necessary, we see that each vertex except for $(0,1)$ may be visited using such a sequence. For the remaining vertex, from the initial configuration $S = \{ (0,0) \} \cup ([k] \times \{ 1\} )$ perform $(0,0) \move (0,2)$, $(i,1) \move (i,0)$ for $i \in [2,k]$, and finally $(1,1) \move (0,1)$. Therefore, $S$ is a mobile general position set.
	\end{proof}

	%The cases $m=2$ and $m=3$ of Theorem~\ref{thm:Hamming-graphs} demonstrate that the bound (i) of Proposition~\ref{prop:bounds-for-cp} is sharp. Consider now $K_n\cp K_n$ with $n\ge 3$. Then $\mob(K_n\cp K_n) = 2n-3$ by Theorem~\ref{thm:Hamming-graphs}. On the other hand, by using~\cite[Theorem 2.3]{tian-2024+} we can conclude that $\gpo(K_n\cp K_n) = n$. This shows that the lower bound~(i) of Proposition~\ref{prop:bounds-for-cp} can be arbitrarily better than the bound~(ii) of the same proposition.
	
	%%%%%%%%%%%%%%%%%%%%%%%%%%%%%%%%%%%%%%%%%%%%%
	\section{Cartesian products with paths}
	\label{sec:prisms}
	%%%%%%%%%%%%%%%%%%%%%%%%%%%%%%%%%%%%%%%%%%%%%
	
	In this section we continue our exposition with exact values of the mobile general position number for some Cartesian products involving paths, including prism graphs, i.e.\ products $G \cp P_2$. The result for prisms of complete graphs is contained in Proposition~\ref{prop:clique cp path} and Theorem~\ref{thm:Hamming-graphs}. We begin with the exact value of $\mob (T \cp K_2)$, where $T$ is a tree. It follows from Proposition~\ref{prop:bounds-for-cp} that for any graph $G$ the mobile general position number of a prism satisfies $\mob (G) \leq \mob (G \cp K_2) \leq 2\gp (G)$. Since the mobile general position number of a tree is just two, Theorem~\ref{thm:prism-tree} shows that $\mob (G \cp K_2)$ can be arbitrarily larger than $\mob (G)$. 
	
	\begin{theorem}
		\label{thm:prism-tree}
		For any tree $T$ with order at least three, $\mob (T \cp K_2) = \ell (T)$.
	\end{theorem}
	
	\begin{proof}
		By Proposition~\ref{prop:clique cp path}, we can assume that $\ell (T) \geq 3$. As remarked in Section~\ref{sec:intro}, $\gpo(T) = \ell (T)$, so by Proposition~\ref{prop:bounds-for-cp}(ii) we have $\mob (T \cp K_2) \geq \ell (T)$. We label the vertices of $K_2$ by $0,1$. Suppose that at least $\ell (T)+1$ robots can traverse $T \cp K_2$ in general position. Since $\gp (T) = \ell (T)$ and each $T$-layer in $T \cp K_2$ is a convex subgraph, neither $T$-layer can contain $> \ell (T)$ robots at any stage, that is, each layer $V(T) \times \{ 0\} $ and $V(T) \times \{ 1\} $ must contain at least one robot at any time. Trivially we can assume that at least one layer contains two or more robots. 
		
		Suppose that the layer $T^1$ contains at least two robots. If not all of the robots in $T^1$ are already stationed at leaves of $T$, then we may suppose that a robot is at a vertex $(u,1)$, where $u$ is a cut-vertex of $T$. Let $T_1,\dots ,T_k$ be the components of $T-u$. As the set of robots is in general position, one of the sets $V(T_i) \times \{ 1\}$ must contain the remaining robots of $T^1$; without loss of generality, suppose that these other robots are in $V(T_1) \times \{ 1\}$. Considering the shortest paths to the robots in $V(T_1) \times \{ 1\} $, we see that there cannot be robots positioned at any vertex from $(\{ u\} \cup \bigcup _{i=2}^kV(T_i)) \times \{ 0\} $. Therefore, the robot at $(u,1)$ can be moved by a sequence of legal moves to $(v,1)$, where $v$ is a leaf of $T$ lying in $T_2$. In this fashion, if both layers $T^i$, $i \in \{0,1\}$, contained at least two robots, then all of these robots could be moved to vertices corresponding to leaves of $T$. However, if $w$ is any leaf of $T$, then we cannot have robots at both $(w,0)$ and $(w,1)$, as this constitutes a maximal general position set of $T \cp K_2$. Therefore, in this case, we conclude that there are at most $\ell (T)$ robots in $T \cp K_2$, a contradiction. 
		
		It follows that there must be a layer, say $T^0$, that contains just one robot $R$, and $\ell (T)$ robots lie in $T^1$, which we can assume to start at the leaves of $T^1$. By the preceding argument, $R$ cannot move to the layer $T^1$ and no robot in $T^1$ can move to $T^0$. If $T$ is a path $P_n$, then we are left with three robots: two positioned at vertices corresponding to leaves of the layer $P_n^1$, and the third robot located at an internal vertex of the path layer $P_n^0$. It is now readily observed that no robot can visit the vertices corresponding to leaves of $P_n^0$. Hence, we may assume that $T$ is not a path. Let $z$ be any vertex of $T$ with degree at least three. No robot in $T^1$ can visit $(z,1)$ without creating three-in-a-line within $T^1$, and robot $R$ cannot leave $T^0$ to visit $(z,1)$, a contradiction. We conclude that $T \cp K_2$ can hold at most $\ell (T)$ robots.   
	\end{proof}

	By Theorem~\ref{thm:prism-tree} we have $\mob(P_n\cp P_2) = 2$ for $n\ge 2$. We next complement this result by considering products of two paths each of order at least three. 
	
	\begin{theorem}
		\label{thm:grids}
		If $n, m\ge 3$, then $\mob(P_n\cp P_m) = 3$.
	\end{theorem}
	
	\begin{proof}
		Let $V(P_k) = [k]$, so that $V(P_n\cp P_m) = [n]\times [m]$. Consider an arbitrary general position set $S$ of $P_n\cp P_m$ with $|S| = 4$. Then from the proof of~\cite[Theorem 2.1]{klavzar-2021} we deduce that none of the corner vertices $(1,1)$, $(1,m)$, $(n,1)$ and $(n,m)$ belongs to $S$. Hence, no sequence of legal moves for any configuration of four robots in general position in $P_n\cp P_m$ can visit any of the vertices $(1,1)$, $(1,m)$, $(n,1)$ and $(n,m)$. Thus, $\mob(P_n\cp P_m) \le 3$.
		
		To prove that $\mob(P_n\cp P_m) \ge 3$, we start with three robots positioned at the general position set $S = \{(1,1), (n,1), (2,m)\}$. We next describe a sequence of legal moves for the three robots. 
		\begin{itemize}
			\item $(1,1)$ moves to all the vertices from $\{ 1\} \times [m-1]$ and returns back to $(1,1)$.  
			\item $(n,1)$ moves to all the vertices from $\{n\} \times [m-1]$ and returns back to $(n,1)$.  
			\item $(2,m)$ moves to all the vertices from $[2, n-1] \times [2, m]$ and returns back to $(2,m)$.  
			\item $(n,1) \move (n,2)$. After this, $(1,1)$ moves to vertices $(2,1), \dots, (n-1,1)$. Notice that at this point the robots are at vertices $(2,m)$, $(n-1,1)$ and $(n,2)$. Moreover, by this stage, all the vertices apart from $(1,m)$ and $(n,m)$ have already been visited.   
			\item $(2,m) \move (1,m)$ and $(n,2) \move(n,3) \move \cdots \move (n,m)$.
		\end{itemize}
		Notice that all these moves are legal, which demonstrates that $\mob(P_n\cp P_m) \ge 3$ and hence 
		$\mob(P_n\cp P_m) = 3$ when $n, m\ge 3$.
	\end{proof}
	
	By contrast, for infinite grids $P_{\infty } \cp P_{\infty }$ we have equality with the general position number.
	
	\begin{theorem}\label{thm:infinitegrid}
		If $P_\infty$ is the two-way infinite path, then $\mob(P_\infty\cp P_\infty) = 4$.
	\end{theorem}
	
	\begin{proof}
		We first recall from \cite[Corollary 3.2]{Manuel-2018b} that $\gp(P_\infty\cp P_\infty) = 4$. Hence, it remains to show that $\mob(P_\infty\cp P_\infty) \ge 4$. To do so, set $V(P_\infty)=\mathbb{Z}$ and let $(i,j)\in V(P_\infty\cp P_\infty)$. We claim that the set $N(i,j) =  \{(i-1,j),(i+1,j),(i,j-1),(i,j+1)\}$ is a mobile general position set of $P_\infty\cp P_\infty$. First, from the proof of \cite[Corollary 3.2]{Manuel-2018b}, we know that any such set $N(i,j)$ is in general position. Next, observe that the sequence of moves $(i+1,j)\move (i+2,j)$, $(i,j+1)\move (i+1,j+1)$, $(i,j-1)\move (i+1,j-1)$ and $(i-1,j)\move (i,j)$ is a sequence of legal moves from robots positioned at the set $N(i,j)$. This sequence moves the four robots one coordinate to the right, leaving robots at $N(i+1,j)$. Fig.~\ref{fig:grids} shows a gp-set of the grid $P_\infty\cp P_\infty$, and one round of moves as just described. The order of the moves is shown by the numeric order in the figure. We can analogously move the four robots in each of the remaining three directions in the infinite grid. In this way, every vertex of the infinite grid is eventually occupied by some robot. Thus, the conclusion follows.
	\end{proof}

	\begin{figure}[ht!]
		\centering
		\begin{tikzpicture}[scale=0.8, vertex_style/.style={circle, ball color=black},vertex_style_1/.style={circle, ball color=red},
			edge_style/.style={thick, black,drop shadow={opacity=0.4}}]
			\def\size{\footnotesize}
			\draw[edge_style] \foreach \x in {1,...,7} \foreach \y in {1,...,6} {
				(\x,\y) -- (\x+1,\y)
			};
			\draw[edge_style] \foreach \x in {1,...,8} \foreach \y in {1,...,5} {
				(\x,\y) -- (\x,\y+1)
			};
			\draw \foreach \x in {1,...,8} \foreach \y in {1,...,6}{
				node[vertex_style] (\x\y) at (\x,\y) {}
			};
			\draw[dotted] \foreach \x in {1,...,8}{
				(\x,0) -- (\x,1)
				(\x,6) -- (\x,7)
			};
			\draw[dotted] \foreach \y in {1,...,6}{
				(0,\y) -- (1,\y)
				(8,\y) -- (9,\y)
			};
			\draw {node[vertex_style_1] at (4,2) {}};
			\draw {node[vertex_style_1] at (4,4) {}};
			\draw {node[vertex_style_1] at (3,3) {}};
			\draw {node[vertex_style_1] at (5,3) {}};
			
			\draw[ultra thick, color=red,->] (53) -- (63);
			\draw[ultra thick, color=red,->] (33) -- (43);
			\draw[ultra thick, color=red,->] (44) -- (54);
			\draw[ultra thick, color=red,->] (42) -- (52);
			
			\draw[thick] {node at (5.4,3.5) {1}};
			\draw[thick] {node at (4.5,2.5) {3}};
			\draw[thick] {node at (4.5,4.5) {2}};
			\draw[thick] {node at (3.5,3.5) {4}};
		\end{tikzpicture}
		\hspace*{0.5cm}
		\begin{tikzpicture}[scale=0.8, vertex_style/.style={circle, ball color=black},vertex_style_1/.style={circle, ball color=red},
			edge_style/.style={thick, black,drop shadow={opacity=0.4}}]
			\def\size{\footnotesize}
			\draw[edge_style] \foreach \x in {1,...,7} \foreach \y in {1,...,6} {
				(\x,\y) -- (\x+1,\y)
			};
			\draw[edge_style] \foreach \x in {1,...,8} \foreach \y in {1,...,5} {
				(\x,\y) -- (\x,\y+1)
			};
			\draw \foreach \x in {1,...,8} \foreach \y in {1,...,6}{
				node[vertex_style] (\x\y) at (\x,\y) {}
			};
			\draw[dotted] \foreach \x in {1,...,8}{
				(\x,0) -- (\x,1)
				(\x,6) -- (\x,7)
			};
			\draw[dotted] \foreach \y in {1,...,6}{
				(0,\y) -- (1,\y)
				(8,\y) -- (9,\y)
			};
			\draw {node[vertex_style_1] at (5,2) {}};
			\draw {node[vertex_style_1] at (5,4) {}};
			\draw {node[vertex_style_1] at (4,3) {}};
			\draw {node[vertex_style_1] at (6,3) {}};
			
		\end{tikzpicture}
		\caption{The legal moves in the infinite grid.}\label{fig:grids}
	\end{figure}

	We now find the exact value of the mobile general position number for some cylinder graphs $C_r \cp P_s$. The general position numbers of the cylinder graphs are given in~\cite{klavzar-2021} as
	
	$$\gp (C_r\cp P_s) = \left\{
	\begin{array}{ll}
		3;     &  r = 3, s = 2, \\
		5;     &  r = 7 \text{ or } r \geq 9, \text{ and } s \geq 5, \\
		4;     & {\rm otherwise}.
	\end{array} \right.
	$$
	Note that Proposition~\ref{prop:clique cp path} gives $\mob (C_3 \cp P_s) = 3$ for $s \geq 2$. We begin with the prism graphs $C_r \cp P_2$.
	
	\begin{theorem}
		\label{thm:prisms-cycles}
		If $n\ge 3$, then 
		$$\mob(C_n\cp K_2) = \left\{
		\begin{array}{ll}
			3;     &  n=3, \\
			2;     &  n=4, \\
			4;     & {\rm otherwise}.
		\end{array} \right.
		$$
	\end{theorem}
	
	\begin{proof}
		The case $C_3\cp K_2 = K_3\cp K_2$ has already been treated above. Up to symmetry, there are unique general position sets of $C_4 \cp K_2$ of cardinalities three and four, both of which are independent sets. However, in both cases no robot can move whilst maintaining the independence property, so that $\mob (C_4 \cp K_2) \leq 2$, and the equality trivially holds.
		
		We assume for the remainder of the proof that $n\ge 5$. It follows from~\cite[Theorem 3.2]{klavzar-2021} that $\gp(C_n\cp K_2) = 4$. Set $V(C_n) = \{ v_i:i \in \mathbb{Z}_n\} $ and $V(K_2) = [2]$. We separate the argument into two cases. 
		
		\medskip
		\noindent 
		{\bf Case 1}: $n$ is odd. \\ 
		Consider a set of four robots located at $S = \{(v_0,1), (v_{\lceil n/2\rceil}, 1), (v_1,2), (v_{\lceil n/2\rceil + 1}, 2)\}\,.$
		Then $S$ is a general position set. Moreover, consider the following sequence of four moves for the robots: 
		\begin{itemize}
			\item $(v_1,2) \move (v_2,2)$;  
			\item $(v_0,1) \move (v_1,1)$; 
			\item $(v_{\lceil n/2\rceil + 1}, 2) \move (v_{\lceil n/2\rceil + 2}, 2)$; 
			\item $(v_{\lceil n/2\rceil}, 1) \move (v_{\lceil n/2\rceil+1}, 1)$. 
		\end{itemize}
		Fig.~\ref{fig:C_5} shows this process for the case $C_5\cp K_2$. 
		
		\begin{figure}[ht!]
			\begin{center}
				\begin{tikzpicture}[scale=0.6,style=thick]
					\tikzstyle{every node}=[draw=none,fill=none]
					\def\vr{3pt} % \vr = vertex radius;  Set \vr = 2/scale for uniform sizing of vertices
					
					\begin{scope}[yshift = 0cm, xshift = 0cm]
						%% vertices defined %%
						\path (0,0.1) coordinate (v1);
						\path (-0.8,1) coordinate (v2);
						\path (-0.5,2) coordinate (v3);
						\path (0.5,2) coordinate (v4);
						\path (0.8,1) coordinate (v5);
						\path (0,-1.2) coordinate (u1);
						\path (-1.8,0.8) coordinate (u2);
						\path (-1.3,2.8) coordinate (u3);
						\path (1.3,2.8) coordinate (u4);
						\path (1.8,0.8) coordinate (u5);
						%% edges %%
						\draw (v1) -- (v2) -- (v3) -- (v4) -- (v5) -- (v1); 
						\draw (u1) -- (u2) -- (u3) -- (u4) -- (u5) -- (u1); 
						\foreach \i in {1,...,5}
						{
							\draw (v\i) -- (u\i);
						}
						\draw [->] (2,1.1) to[bend right=30] node {} (1.6,2.7);
						%% vertices %%%
						\foreach \i in {1,...,5}
						{
							\draw (v\i)  [fill=white] circle (\vr);
							\draw (u\i)  [fill=white] circle (\vr);
						}
						\draw (v1)  [fill=black] circle (\vr);
						\draw (v3)  [fill=black] circle (\vr);
						\draw (u2)  [fill=black] circle (\vr);
						\draw (u5)  [fill=black] circle (\vr);
						%% text %%
						%\draw[below] (g1)++(0.0,-0.1) node {$g_1$};
					\end{scope}
					
					%%%% Second
					
					\begin{scope}[yshift = 0cm, xshift = 5cm]
						%% vertices defined %%
						\path (0,0.1) coordinate (v1);
						\path (-0.8,1) coordinate (v2);
						\path (-0.5,2) coordinate (v3);
						\path (0.5,2) coordinate (v4);
						\path (0.8,1) coordinate (v5);
						\path (0,-1.2) coordinate (u1);
						\path (-1.8,0.8) coordinate (u2);
						\path (-1.3,2.8) coordinate (u3);
						\path (1.3,2.8) coordinate (u4);
						\path (1.8,0.8) coordinate (u5);
						%% edges %%
						\draw (v1) -- (v2) -- (v3) -- (v4) -- (v5) -- (v1); 
						\draw (u1) -- (u2) -- (u3) -- (u4) -- (u5) -- (u1); 
						\foreach \i in {1,...,5}
						{
							\draw (v\i) -- (u\i);
						}
						\draw [->] (0.25,0.1) to[bend right=30] node {} (0.8,0.7);
						%% vertices %%%
						\foreach \i in {1,...,5}
						{
							\draw (v\i)  [fill=white] circle (\vr);
							\draw (u\i)  [fill=white] circle (\vr);
						}
						\draw (v1)  [fill=black] circle (\vr);
						\draw (v3)  [fill=black] circle (\vr);
						\draw (u2)  [fill=black] circle (\vr);
						\draw (u4)  [fill=black] circle (\vr);
						%% text %%
						%\draw[below] (g1)++(0.0,-0.1) node {$g_1$};
					\end{scope}
					
					%%% Third
					
					\begin{scope}[yshift = 0cm, xshift = 10cm]
						%% vertices defined %%
						\path (0,0.1) coordinate (v1);
						\path (-0.8,1) coordinate (v2);
						\path (-0.5,2) coordinate (v3);
						\path (0.5,2) coordinate (v4);
						\path (0.8,1) coordinate (v5);
						\path (0,-1.2) coordinate (u1);
						\path (-1.8,0.8) coordinate (u2);
						\path (-1.3,2.8) coordinate (u3);
						\path (1.3,2.8) coordinate (u4);
						\path (1.8,0.8) coordinate (u5);
						%% edges %%
						\draw (v1) -- (v2) -- (v3) -- (v4) -- (v5) -- (v1); 
						\draw (u1) -- (u2) -- (u3) -- (u4) -- (u5) -- (u1); 
						\foreach \i in {1,...,5}
						{
							\draw (v\i) -- (u\i);
						}
						\draw [->] (-1.85,0.5) to[bend right=30] node {} (-0.3,-1.2);
						%% vertices %%%
						\foreach \i in {1,...,5}
						{
							\draw (v\i)  [fill=white] circle (\vr);
							\draw (u\i)  [fill=white] circle (\vr);
						}
						\draw (v5)  [fill=black] circle (\vr);
						\draw (v3)  [fill=black] circle (\vr);
						\draw (u2)  [fill=black] circle (\vr);
						\draw (u4)  [fill=black] circle (\vr);
						%% text %%
						%\draw[below] (g1)++(0.0,-0.1) node {$g_1$};
					\end{scope}
					
					%% Fourth
					
					\begin{scope}[yshift = 0cm, xshift = 15cm]
						%% vertices defined %%
						\path (0,0.1) coordinate (v1);
						\path (-0.8,1) coordinate (v2);
						\path (-0.5,2) coordinate (v3);
						\path (0.5,2) coordinate (v4);
						\path (0.8,1) coordinate (v5);
						\path (0,-1.2) coordinate (u1);
						\path (-1.8,0.8) coordinate (u2);
						\path (-1.3,2.8) coordinate (u3);
						\path (1.3,2.8) coordinate (u4);
						\path (1.8,0.8) coordinate (u5);
						%% edges %%
						\draw (v1) -- (v2) -- (v3) -- (v4) -- (v5) -- (v1); 
						\draw (u1) -- (u2) -- (u3) -- (u4) -- (u5) -- (u1); 
						\foreach \i in {1,...,5}
						{
							\draw (v\i) -- (u\i);
						}
						\draw [->] (-0.75,1.9) to[bend right=30] node {} (-0.9,1.2);
						%% vertices %%%
						\foreach \i in {1,...,5}
						{
							\draw (v\i)  [fill=white] circle (\vr);
							\draw (u\i)  [fill=white] circle (\vr);
						}
						\draw (v5)  [fill=black] circle (\vr);
						\draw (v3)  [fill=black] circle (\vr);
						\draw (u1)  [fill=black] circle (\vr);
						\draw (u4)  [fill=black] circle (\vr);
						%% text %%
						%\draw[below] (g1)++(0.0,-0.1) node {$g_1$};
					\end{scope}
					
					%%% Last one
					
					\begin{scope}[yshift = 0cm, xshift = 20cm]
						%% vertices defined %%
						\path (0,0.1) coordinate (v1);
						\path (-0.8,1) coordinate (v2);
						\path (-0.5,2) coordinate (v3);
						\path (0.5,2) coordinate (v4);
						\path (0.8,1) coordinate (v5);
						\path (0,-1.2) coordinate (u1);
						\path (-1.8,0.8) coordinate (u2);
						\path (-1.3,2.8) coordinate (u3);
						\path (1.3,2.8) coordinate (u4);
						\path (1.8,0.8) coordinate (u5);
						%% edges %%
						\draw (v1) -- (v2) -- (v3) -- (v4) -- (v5) -- (v1); 
						\draw (u1) -- (u2) -- (u3) -- (u4) -- (u5) -- (u1); 
						\foreach \i in {1,...,5}
						{
							\draw (v\i) -- (u\i);
						}
						%\draw [->] (-0.75,1.9) to[bend right=30] node {} (-0.9,1.2);
						%% vertices %%%
						\foreach \i in {1,...,5}
						{
							\draw (v\i)  [fill=white] circle (\vr);
							\draw (u\i)  [fill=white] circle (\vr);
						}
						\draw (v5)  [fill=black] circle (\vr);
						\draw (v2)  [fill=black] circle (\vr);
						\draw (u1)  [fill=black] circle (\vr);
						\draw (u4)  [fill=black] circle (\vr);
						%% text %%
						%\draw[below] (g1)++(0.0,-0.1) node {$g_1$};
					\end{scope}
					
				\end{tikzpicture}
			\end{center}
			\caption{Moving robots in $C_5\cp K_2$}
			\label{fig:C_5}
		\end{figure}
		
		\medskip
		\noindent 
		{\bf Case 2}: $n$ is even. \\ 
		Suppose now that the robots are positioned at $S = \{(v_0,1), (v_{n/2}, 1), (v_1,2), (v_{n/2 + 1}, 2)\}\,.$
		Again $S$ is a general position set. Moreover, consider the following sequence of moves: 
		\begin{itemize}
			\item $(v_1,2) \move (v_2,2)$;  
			\item $(v_{n/2 + 1}, 2) \move (v_{n/2 + 2}, 2)$; 
			\item $(v_0,1) \move (v_1,1)$; 
			\item $(v_{n/2}, 1) \move (v_{n/2+1}, 1)$. 
		\end{itemize}
		Fig.~\ref{fig:C_6} shows this process for the case $C_6\cp K_2$. 
		
		\begin{figure}[ht!]
			\begin{center}
				\begin{tikzpicture}[scale=0.6,style=thick]
					\tikzstyle{every node}=[draw=none,fill=none]
					\def\vr{3pt} % \vr = vertex radius;  Set \vr = 2/scale for uniform sizing of vertices
					
					\begin{scope}[yshift = 0cm, xshift = 0cm]
						%% vertices defined %%
						\path (0,0) coordinate (v1);
						\path (-0.5,1) coordinate (v2);
						\path (-0.5,2) coordinate (v3);
						\path (0,3) coordinate (v4);
						\path (0.5,2) coordinate (v5);
						\path (0.5,1) coordinate (v6);
						\path (0,-1) coordinate (u1);
						\path (-1.5,0.7) coordinate (u2);
						\path (-1.5,2.3) coordinate (u3);
						\path (0,4) coordinate (u4);
						\path (1.5,2.3) coordinate (u5);
						\path (1.5,0.7) coordinate (u6);
						%% edges %%
						\draw (v1) -- (v2) -- (v3) -- (v4) -- (v5) -- (v6) -- (v1); 
						\draw (u1) -- (u2) -- (u3) -- (u4) -- (u5) -- (u6) -- (u1); 
						\foreach \i in {1,...,6}
						{
							\draw (v\i) -- (u\i);
						}
						\draw [->] (1.7,0.9) to[bend right=30] node {} (1.7,2.1);
						%% vertices %%%
						\foreach \i in {1,...,6}
						{
							\draw (v\i)  [fill=white] circle (\vr);
							\draw (u\i)  [fill=white] circle (\vr);
						}
						\draw (v1)  [fill=black] circle (\vr);
						\draw (v4)  [fill=black] circle (\vr);
						\draw (u3)  [fill=black] circle (\vr);
						\draw (u6)  [fill=black] circle (\vr);
						%% text %%
						%\draw[below] (g1)++(0.0,-0.1) node {$g_1$};
					\end{scope}
					
					%%% second copy
					
					\begin{scope}[yshift = 0cm, xshift = 5cm]
						%% vertices defined %%
						\path (0,0) coordinate (v1);
						\path (-0.5,1) coordinate (v2);
						\path (-0.5,2) coordinate (v3);
						\path (0,3) coordinate (v4);
						\path (0.5,2) coordinate (v5);
						\path (0.5,1) coordinate (v6);
						\path (0,-1) coordinate (u1);
						\path (-1.5,0.7) coordinate (u2);
						\path (-1.5,2.3) coordinate (u3);
						\path (0,4) coordinate (u4);
						\path (1.5,2.3) coordinate (u5);
						\path (1.5,0.7) coordinate (u6);
						%% edges %%
						\draw (v1) -- (v2) -- (v3) -- (v4) -- (v5) -- (v6) -- (v1); 
						\draw (u1) -- (u2) -- (u3) -- (u4) -- (u5) -- (u6) -- (u1); 
						\foreach \i in {1,...,6}
						{
							\draw (v\i) -- (u\i);
						}
						%\draw (g1-1) to[bend left=40] node {} (g1-3);
						%% vertices %%%
						\foreach \i in {1,...,6}
						{
							\draw (v\i)  [fill=white] circle (\vr);
							\draw (u\i)  [fill=white] circle (\vr);
						}
						\draw (v1)  [fill=black] circle (\vr);
						\draw (v4)  [fill=black] circle (\vr);
						\draw (u3)  [fill=black] circle (\vr);
						\draw (u5)  [fill=black] circle (\vr);
						\draw [->] (-1.7,2.1) to[bend right=30] node {} (-1.7,0.9);
						
						%% text %%
						%\draw[below] (g1)++(0.0,-0.1) node {$g_1$};
					\end{scope}
					
					%%% third copy
					
					\begin{scope}[yshift = 0cm, xshift = 10cm]
						%% vertices defined %%
						\path (0,0) coordinate (v1);
						\path (-0.5,1) coordinate (v2);
						\path (-0.5,2) coordinate (v3);
						\path (0,3) coordinate (v4);
						\path (0.5,2) coordinate (v5);
						\path (0.5,1) coordinate (v6);
						\path (0,-1) coordinate (u1);
						\path (-1.5,0.7) coordinate (u2);
						\path (-1.5,2.3) coordinate (u3);
						\path (0,4) coordinate (u4);
						\path (1.5,2.3) coordinate (u5);
						\path (1.5,0.7) coordinate (u6);
						%% edges %%
						\draw (v1) -- (v2) -- (v3) -- (v4) -- (v5) -- (v6) -- (v1); 
						\draw (u1) -- (u2) -- (u3) -- (u4) -- (u5) -- (u6) -- (u1); 
						\foreach \i in {1,...,6}
						{
							\draw (v\i) -- (u\i);
						}
						%\draw (g1-1) to[bend left=40] node {} (g1-3);
						%% vertices %%%
						\foreach \i in {1,...,6}
						{
							\draw (v\i)  [fill=white] circle (\vr);
							\draw (u\i)  [fill=white] circle (\vr);
						}
						\draw (v1)  [fill=black] circle (\vr);
						\draw (v4)  [fill=black] circle (\vr);
						\draw (u2)  [fill=black] circle (\vr);
						\draw (u5)  [fill=black] circle (\vr);
						\draw [->] (0.3,0.1) to[bend right=30] node {} (0.6,0.75);
						%% text %%
						%\draw[below] (g1)++(0.0,-0.1) node {$g_1$};
					\end{scope}
					
					%%% fourth copy
					
					\begin{scope}[yshift = 0cm, xshift = 15cm]
						%% vertices defined %%
						\path (0,0) coordinate (v1);
						\path (-0.5,1) coordinate (v2);
						\path (-0.5,2) coordinate (v3);
						\path (0,3) coordinate (v4);
						\path (0.5,2) coordinate (v5);
						\path (0.5,1) coordinate (v6);
						\path (0,-1) coordinate (u1);
						\path (-1.5,0.7) coordinate (u2);
						\path (-1.5,2.3) coordinate (u3);
						\path (0,4) coordinate (u4);
						\path (1.5,2.3) coordinate (u5);
						\path (1.5,0.7) coordinate (u6);
						%% edges %%
						\draw (v1) -- (v2) -- (v3) -- (v4) -- (v5) -- (v6) -- (v1); 
						\draw (u1) -- (u2) -- (u3) -- (u4) -- (u5) -- (u6) -- (u1); 
						\foreach \i in {1,...,6}
						{
							\draw (v\i) -- (u\i);
						}
						%\draw (g1-1) to[bend left=40] node {} (g1-3);
						%% vertices %%%
						\foreach \i in {1,...,6}
						{
							\draw (v\i)  [fill=white] circle (\vr);
							\draw (u\i)  [fill=white] circle (\vr);
						}
						\draw (v6)  [fill=black] circle (\vr);
						\draw (v4)  [fill=black] circle (\vr);
						\draw (u2)  [fill=black] circle (\vr);
						\draw (u5)  [fill=black] circle (\vr);
						\draw [->] (-0.2,3.0) to[bend right=30] node {} (-0.6,2.2);
						%% text %%
						%\draw[below] (g1)++(0.0,-0.1) node {$g_1$};
					\end{scope}
					
					%%% fifth copy
					
					\begin{scope}[yshift = 0cm, xshift = 20cm]
						%% vertices defined %%
						\path (0,0) coordinate (v1);
						\path (-0.5,1) coordinate (v2);
						\path (-0.5,2) coordinate (v3);
						\path (0,3) coordinate (v4);
						\path (0.5,2) coordinate (v5);
						\path (0.5,1) coordinate (v6);
						\path (0,-1) coordinate (u1);
						\path (-1.5,0.7) coordinate (u2);
						\path (-1.5,2.3) coordinate (u3);
						\path (0,4) coordinate (u4);
						\path (1.5,2.3) coordinate (u5);
						\path (1.5,0.7) coordinate (u6);
						%% edges %%
						\draw (v1) -- (v2) -- (v3) -- (v4) -- (v5) -- (v6) -- (v1); 
						\draw (u1) -- (u2) -- (u3) -- (u4) -- (u5) -- (u6) -- (u1); 
						\foreach \i in {1,...,6}
						{
							\draw (v\i) -- (u\i);
						}
						%\draw (g1-1) to[bend left=40] node {} (g1-3);
						%% vertices %%%
						\foreach \i in {1,...,6}
						{
							\draw (v\i)  [fill=white] circle (\vr);
							\draw (u\i)  [fill=white] circle (\vr);
						}
						\draw (v6)  [fill=black] circle (\vr);
						\draw (v3)  [fill=black] circle (\vr);
						\draw (u2)  [fill=black] circle (\vr);
						\draw (u5)  [fill=black] circle (\vr);
						%% text %%
						%\draw[below] (g1)++(0.0,-0.1) node {$g_1$};
					\end{scope}
					
				\end{tikzpicture}
			\end{center}
			\caption{Moving robots in $C_6\cp K_2$}
			\label{fig:C_6}
		\end{figure}
		
		In both cases above, we note that these four moves are legal. Since the obtained sets are symmetric with respect to the original ones, by repeating these procedures the robots will eventually visit all the vertices of $C_n\cp K_2$. It follows that each $S$ is a mobile general position set, and hence $\mob(C_n\cp K_2) \ge 4$.   
	\end{proof}
	
	We now introduce a technical lemma that allows us to extend results on short cylinders to longer cylinders.
	
	\begin{lemma}\label{lem:wanderingcylinder}
		If $H$ is a connected graph with girth at least $2r$ and radius at least $r-1$, then for any graph $G$ it holds that \[ \mob (G \cp H) \geq \mob (G \cp P_r).\]
	\end{lemma}
	
	\begin{proof}
		Let $Q$ be any path $v_1,v_2,\dots ,v_r$ of length $r-1$ in $H$. The subgraph of $G \cp H$ induced by $V(G) \times V(Q) $ is isomorphic to $G \cp P_r$, and as the girth of $H$ is at least $2r$, the subgraph is convex. Thus, $\mob (G \cp P_r)$ robots can traverse the vertices of $V(G) \times V(Q) $ in general position without leaving the subgraph. Now, let $v_{r+1} \in N_H(v_r) \setminus \{ v_{r-1}\} $. Suppose that the robots have visited  all the vertices of $V(G) \times V(Q)$ by a sequence of legal moves. Next, move all robots in $V(G) \times \{ v_r\} $ to $V(G) \times \{ v_{r+1}\} $ in turn by legal moves of the form $(u,v_r) \move (u,v_{r+1})$, where $u \in V(G)$. Then repeat this process to move the robots in $V(G) \times \{ v_j\} $ to $V(G) \times \{ v_{j+1}\} $ for $j = r-1,r-2,\dots ,1$. As $H$ has girth at least $2r$, the robots remain in general position throughout this process. As $H$ is connected and the radius of $H$ is at least $r-1$, any layer $G^h$ can be visited in this way. 
	\end{proof}
	
	Notice that Lemma~\ref{lem:wanderingcylinder} generalises the inequality $\mob (G) \leq \mob (G \cp K_2)$. We first focus on cylinders with cycles of length four.
	
	\begin{proposition}
		If $s \geq 3$ is an integer, then $\mob (C_4 \cp P_s) = 3$.
	\end{proposition}
	
	\begin{proof}
		Let $s \geq 3$, $V(C_4) = \mathbb{Z}_4$, and $V(P_s) = [s]$. We first show that $\mob(C_4 \cp P_s) \leq 3$. Suppose for a contradiction that there exists a mobile general position set $S$ of $C_4 \cp P_s$ with $|S| \geq 4$. 
		
		Clearly, no three robots from $S$ can lie in the same $P_s$-layer. Suppose that two robots $R_1$ and $R_2$ in $S$ lie in the same $P_s$-layer of $C_4 \cp P_s$; without loss of generality, $R_1$ and $R_2$ are stationed at vertices $(0,i)$ and $(0,j)$ respectively, where $1 \leq i < j \leq s$. There is a shortest path from any vertex $(u,v)$ with $v \in [i]$ to $R_2$ through $R_1$, and likewise for vertices with second coordinate at 
		least $j$. Hence, there are no robots on $\mathbb{Z}_4 \times \left ( [1,i] \cup [j,s]\right )$ apart from $R_1$ and $R_2$. Call the other two robots $R_3$ and $R_4$. By the above analysis, $R_3$ and $R_4$ cannot have the same first coordinate. Hence we can assume that robot $R_3$ lies at $(1,k)$, where $i < k < j$. Any vertex in $\left (\{ 1,2\} \times [i+1,j-1] \right ) \setminus \{ (1,k)\} $ has a shortest path to either $R_1$ or $R_2$ through $R_3$, so $R_4$ must be at a vertex $(3,l)$, where $i < l < j$. However, in this configuration each robot is only free to move within its $P_s$-layer, and so no robot can visit any vertex in $^2P_s$, a contradiction. Consequently, there is one robot on each $P_s$-layer, and none of the robots can move to another $P_s$-layer.
		
		Observe that any pair of adjacent vertices constitutes a maximal general position set of $C_4 \cp P_s$, so the robots must at all times occupy an independent set. Therefore, if we suppose that $R_1$ is the robot located at the vertex with smallest second coordinate in the initial configuration, say at vertex $(0,i)$, then no robot in the layers $^1P_s$ or $^3P_s$ can ever move to a position with second coordinate smaller than that of $R_1$, and hence the vertices in $\{ 1,3\} \times \{ 1\} $ cannot be visited by legal moves. Thus, $\mob (C_4 \cp P_s) \leq 3$.
		
		To show the lower bound, consider $C_4 \cp P_3$. We start with robots at vertices $(0,1)$, $(1,2)$ and $(0,3)$. Then for $i = 0,1,2$ in succession we perform the sequence of three legal moves $(i+1,2) \move (i+2,2)$, $(i,3) \move (i+1,3)$ and $(i,1) \move (i+1,1)$ in this order. Lemma~\ref{lem:wanderingcylinder} now gives the result for cylinders $C_4 \cp P_s$ for $s \geq 4$.    
	\end{proof}
	
	Theorem~\ref{thm:prisms-cycles} for prisms along with Lemma~\ref{lem:wanderingcylinder} implies that $\mob (C_r \cp P_s) \geq 4$ for $r \geq 5, s \geq 2$. Combined with the upper bound involving $\gp (C_r \cp P_s)$ we see that $\mob (C_r \cp P_s) = 4$ for $r \in \{ 5,6,8\} $ and $s \geq 2$, or for $r \geq 5$ and $s \leq 4$. In the cases $r = 7$ or $r \geq 9$ and $s \geq 5$ the mobile general position number of $C_r \cp P_s$ must be either four or five.
	
	\begin{proposition}\label{prop:cylinders}
		If $r = 9$ or $r \geq 11$, and $s \geq 5$, then $\mob(C_r \cp P_s) = 5$. 
	\end{proposition}
	
	\begin{proof}
		It is known from~\cite[Theorem 3.2]{klavzar-2021} that $\gp (C_r \cp P_s) = 5$ for $r = 9$ or $r \geq 11$, and $s \geq 5$. It only remains to show that $\mob(C_r \cp P_s) \geq 5$. We show that the result is true for $C_r \cp P_5$, and the full claim then follows for larger values of $s$ by Lemma~\ref{lem:wanderingcylinder}. Set $V(C_r) = \mathbb{Z}_r$ and $V(P_5) = [5]$. 
		
		If $r \geq 11$ and $i \in \mathbb{Z}_r$, then we consider the set
		\[ S_{i,0}  = \{(i+1,1),(i+4,2),(i+\lfloor{r/2}\rfloor+2,3),(i,4),(i+3,5)\} .\]
		For $i = 0,1,\dots ,r-2$ we define the following sequence of five moves:
		\begin{itemize}
			\item $(i+4,2) \move (i+5,2)$ to give $S_{i,1}$,
			\item $(i+3,5) \move (i+4,5)$ to give $S_{i,2}$,
			\item $(i+\lfloor{r/2}\rfloor+2,3) \move (i+\lfloor{r/2}\rfloor+3,3)$ to give $S_{i,3}$,
			\item $(i+1,1) \move (i+2,1)$ to give $S_{i,4}$,
			\item $(i,4) \move (i+1,4)$.
		\end{itemize}
		The final move brings us to the configuration $S_{i+1,0}$. We start with the five robots positioned at the set $S_{0,0}$ and perform these sequences of moves for $i = 0,1,\dots r-2$. Each of these moves is legal. To see this, notice that the robots remain in general position at each stage, which is easily verified for $S_{0,j}$, $j \in [5]$, and by then observing that the automorphism that maps $(u,v)$ to $(u+i,v)$ for all $u \in \mathbb{Z}_r,s \in [5]$ transforms $S_{0,j}$ to $S_{i,j}$ for $j \in [5]$. Moreover, by the end of the process, all vertices have been visited.
		
		Similarly, for $r = 9$ we start with robots positioned at the set 
		\[ \{(1,1),(4,2),(\lfloor{s/2}\rfloor+2,3),(0,4),(3,5)\}\]
		and perform the sequence of moves:
		\begin{itemize}
			\item $(\lfloor{s/2}\rfloor+2,3) \move (\lfloor{s/2}\rfloor+3,3)$,
			\item $(1,1) \move (2,1)$,
			\item $(0,4) \move (1,4)$,
			\item $(4,2) \move (5,2)$,
			\item $(3,5) \move (4,5)$.
		\end{itemize}
		By repeating these moves the robots visit all of the vertices of $C_9 \cp P_5$ by legal moves.
	\end{proof}
	
	Hence, the only unknown values are $\mob (C_7 \cp P_s)$ and $\mob (C_{10} \cp P_s)$ for $s \geq 5$. We conjecture that the answer is four in these cases.
	
	By combining Lemma~\ref{lem:wanderingcylinder} with Proposition~\ref{prop:cylinders} we obtain a lower bound for the mobile general position number of sufficiently large torus graphs.
	\begin{corollary}
		\label{cor:Cr cp Cs}
		For $r = 9$ or $r > 10$, and $s \geq 10$, $\mob (C_r \cp C_s) \geq 5$.
	\end{corollary}
	It is shown in~\cite{KorzeVesel} that if $r,s \geq 7$ and $r$ and $s$ do not both lie in $\{ 8,10,12\}$, then $\gp (C_r \cp C_s) = 7$. Computer search shows that the torus $C_9 \cp C_8$ has mobile general position number seven, so this upper bound can be achieved. 
	
	%%%%%%%%%%%%%%%%%%%%%%%%%%%%%%%%%%%%%%%%%%%%%
	\section{Corona products and joins}
	\label{sec:corona+joins}
	%%%%%%%%%%%%%%%%%%%%%%%%%%%%%%%%%%%%%%%%%%%%%
	
	In this section we consider moving robots in general position through corona products and joins. We first define these two graph operations. 
	
	Given two graphs $G$ and $H$ with $V(G) = \{v_1, \ldots ,v_{n}\}$, the {\em corona product} graph $G\odot H$ is formed by taking one copy of $G$ and $n$ disjoint copies of $H$, call them $H^1,\ldots, H^{n}$, and for each $i \in [n]$ adding all the possible edges between $v_i\in V(G)$ and every vertex of $H^i$. For $i \in [n]$ we will write $\widetilde{H}_i$ for the subgraph of $G \odot H$ induced by $V(H^i) \cup \{ v_i\}$. Also, the {\em join} $G\vee H$ of graphs $G$ and $H$ is obtained from the disjoint union of $G$ and $H$ by adding all possible edges between $G$ and $H$.
	
	\subsection{Corona product graphs}
	
	The first paper~\cite{klavzar-2023} on the mobile general position problem briefly considered mobile general position sets in rooted products. This suggests investigating the problem in corona products, which can also be viewed as a kind of rooted product. The general position number of corona product graphs was studied in~\cite{ghorbani-2021}. We now bound the value of the mobile general position number of the corona product $G\odot H$. 
	
	\begin{theorem}
		\label{thm:corona-bounds}
		For any two graphs $G$ and $H$,
		$$\max \{ \mob(G), \mob(H \vee K_1) \} \le \mob(G\odot H) \le \max \{n(G), \gp(H \vee K_1)\}\,.$$
	\end{theorem}
	
	\begin{proof}
		Let $S$ be a mobile general position set of $H\vee K_1$ and let $S_1$ be its copy in $\widetilde{H}_1$. We claim that $S_1$ is a mobile general position set of $G\odot H$. As $\widetilde{H}_1$ is an isometric subgraph of $G \odot H$, first the robots from $S_1$ can visit each vertex of $\widetilde{H}_1$. Next, as soon as one robot visits the vertex $v_1$, this robot can visit all the vertices of $V(G\odot H)\setminus V(\widetilde{H}_1)$ before returning to $v_1$. It follows that $\mob(G\odot H) \ge \mob(H \vee K_1)$. 
		
		Now, let $S$ be a mobile general position set of $G$ and let $S^{\prime }$ be the copy of $S$ in $G\odot H$. Then each vertex $v_i\in V(G)$ can be visited by a robot from $S^{\prime }$. Moreover, as soon as a robot moves to some vertex $v_i$, this robot can visit all the vertices from $H^i$ and then return to $v_i$. Hence $S^{\prime }$ is a mobile general position set of $G\odot H$ and $\mob(G\odot H) \ge \mob(G)$. We conclude that $\mob(G\odot H) \ge \max\{ \mob(H \vee K_1), \mob(G)\}$. 
		
		To prove the upper bound, let $S$ be a mobile general position set of $G\odot H$. If $|S|\le n(G)$, then there is nothing to prove. Assume next that $|S| \ge n(G)+1$. Then by the pigeonhole principle we have $|S \cap V(\widetilde{H}_i)| \ge 2$ for some $i\in [n(G)]$. Hence either at some point there is already a robot in $H^i$ and a second robot enters $\widetilde{H}_i$ via $v_i$, or else there are always at least two robots in $V(H_i)$ and a further robot must visit $v_i$. Denote the positions of the robots at this moment by $S^{\prime }$. In $S^{\prime }$ there is a robot $R_1$ in $V(H^i)$ and a robot $R_2$ at the cut-vertex $v_i$. As any path from $R_1$ to a robot on $V(G \odot H) \setminus V(\widetilde{H}_i)$ would pass through $R_2$, it follows that $S^{\prime } \subseteq V(\widetilde{H}_i)$ and $S^{\prime }$ is a general position set of $\widetilde{H}_i$. Hence, under the assumption that $|S| \ge n(G)+1$, we must have $\mob(G\odot H) = |S^{\prime }| \le \gp(H \vee K_1)$. 
	\end{proof}
	
	When both $G$ and $H$ are complete graphs, $G\odot H$ is a block graph, hence the following consequence can also be deduced from~\cite[Theorem~2.3]{klavzar-2023}.  
	
	\begin{corollary}
		\label{cor:complete-complete}
		If $r,s \ge 1$, then
		$\mob(K_r\odot K_s) = \max\{r,s+1\}$.
	\end{corollary}
	
	Note that Corollary~\ref{cor:complete-complete} demonstrates the sharpness of all the bounds in Theorem~\ref{thm:corona-bounds}. For another sharpness example, in which the upper and lower bounds do not coincide, consider $G = K_2$ and $H = C_4$. Then we have that $\mob(K_2)=2$, $\mob(C_4\vee K_1)=2$, and $\gp(C_4\vee K_1)=3$. It can be noted that $\mob(K_2\odot C_4) = 3 = \gp (C_4 \vee K_1)$. For another infinite family, let $G$ be an arbitrary tree, and $H$ the edgeless graph of order at least two. Since in this case $G\odot H$ is a tree, we have $\mob(G\odot H) = 2$ (see~\cite[Theorem 2.3]{klavzar-2023}), which is also the value of the lower bound in Theorem~\ref{thm:corona-bounds}. 
	
	We complete this section by presenting a result which shows that none of the bounds of Theorem~\ref{thm:corona-bounds} is sharp in general. 
	
	\begin{theorem}
		\label{thm:corona-cycle-K_1}    
		If $n\ge 3$, then $\mob(C_n \odot K_1) = \left\lceil \frac{n}{2}\right\rceil + 1$. 
	\end{theorem}
	
	\begin{proof}
		Set $V(C_n) = \mathbb{Z}_n$ and, for each $i\in \mathbb{Z}_n$, let $i'$ be the leaf in $C_n\odot K_1$ attached to $i$. We first show that $\mob(C_n \odot K_1) \ge \left\lceil \frac{n}{2}\right\rceil + 1$. Set $Y_i = \{i^{\prime },(i+1)^{\prime },\dots,\left ( \left\lceil \frac{n}{2}\right\rceil +i\right )^{\prime }\}$ for any $i \in \mathbb{Z}_n$. We claim that $Y_0$ is a mobile general position set. It is easily seen that each $Y_i$ is a general position set of $C_n \odot K_1$. 
		
		Starting from some fixed $Y_i$, we move the robot $R$ placed at $\left ( \left\lceil \frac{n}{2}\right\rceil +i\right )^{\prime }$ through the vertices of the set $[ \left ( \left\lceil \frac{n}{2}\right\rceil +i\right ),(i-1)]$ and their attached leaves, and finally leave the robot at $(i-1)^{\prime }$. After these moves we are left with the robots occupying the set $Y_{i-1}$ and at each stage the robots remained in general position, since the paths between robots in $Y_i \setminus \{ \left (\left \lceil \frac{n}{2} \right \rceil +i \right )^{\prime }\}$ pass through $[i,\left \lceil \frac{n}{2}\right \rceil +i-1]$. Therefore starting at $Y_0$ and repeating this process for $i = 0,-1,\dots ,-\left \lceil \frac{n}{2} \right \rceil $ results in all vertices being visited by a robot and $Y_0$ is a mobile general position set as claimed.
		
		We now prove that $\mob(C_n \odot K_1) \le \left\lceil \frac{n}{2}\right\rceil + 1$. Note that each set $\{ i,i^{\prime }\} $, $i \in \mathbb{Z}_n$, is a maximal general position set, so we must have $|X\cap \{i,i'\}| \le 1$ for each $i\in \mathbb{Z}_n$ and any $\mob$-set $X$. Consider the moment that a robot visits the vertex $0$. Let $j,k$ be the smallest and largest values in $[1,n-1]$, respectively, such that there is a robot in $\{ j,j^{\prime }\}$ and $\{ k,k^{\prime }\}$. To avoid the robot in $\{ k,k^{\prime }\} $ having a shortest path to the robot in $\{ j,j^{\prime }\}$ through the robot at $0$ we must have $k-j \leq \left \lceil \frac{n}{2} \right \rceil -1$. As each set $\{ i,i^{\prime }\}$ contains at most one robot for $i \in [j,k] \cup \{ 0\}$ and no robots for $i \in [k+1,n-1] \cup [1,j-1]$, this gives an upper bound of $\left \lceil \frac{n}{2} \right \rceil +1$ robots in the graph.  
	\end{proof}
	
	Notice that for $C_n \odot K_1$, the lower bound of Theorem~\ref{thm:corona-bounds} is three if $n \geq 3$ and $n \notin \{ 4,6\} $, and it is two if $n \in \{ 4,6\}$, whereas the upper bound is $n$. Since $C_n \odot K_1$ is a unicyclic graph, we may recall that mobile general position sets of unicyclic graphs were discussed in~\cite{klavzar-2023}.
	
	\subsection{Joins of graphs}
	We now give bounds for the mobile general position number of joins $G \vee H$. Observe that if both $G$ and $H$ are cliques, then $G \vee H$ is also a clique and the question is trivial, so we will assume that at least one of $G$ and $H$ is not a clique.
	
	\begin{theorem}
		\label{thm:join}
		If $G$ and $H$ are (not necessarily connected) graphs with clique number at least two, and $G$ and $H$ are not both cliques, then \[ \min \{ \omega(G),\omega (H)\} +1 \leq \mob (G \vee H) \leq \omega (G)+\omega (H)-1.\] For any graph $G$ with order $n \geq 2$, \[  2 \leq \mob (G \vee K_1) \leq \omega (G)+1.\] 
	\end{theorem}
	
	\begin{proof}
		Assume that both $G$ and $H$ have clique number at least two and that at least two robots are traversing $G \vee H$ in general position. At some point there must be a robot in $G$ and a robot in $H$. Hence, at this point, the set of occupied vertices in $G$ and the occupied vertices in $H$ must both be cliques in $G$ and $H$, respectively, giving the upper bound $\mob (G \vee H) \leq \omega (G) + \omega (H)$. However, if $\mob (G \vee H) = \omega (G) + \omega (H)$, then no robot has a legal move, since any move would result in a clique in one of $G$ and $H$ and a non-clique in the other, so in fact $\mob (G \vee H) \leq \omega (G) + \omega (H)-1$.
		
		For the lower bound, assume that $\omega (H) \leq \omega (G)$. We can start with robots at a maximum clique $W_H$ of $H$ and one robot in $G$. The robot in $G$ can visit every vertex of $G$, since during this process the occupied vertices always form a clique in $G \vee H$. At the end, this robot moves into a maximum clique $W_G$ of $G$. After that, all the robots from $W_H$ but one move into $W_G$. At that time, only one robot remains in $H$ and, by the same argument, it can visit every vertex of $H$. 
		
		The inequalities for $G \vee K_1$ can be derived in a similar manner.
	\end{proof}
	
	Note that if $G$ and $H$ both have clique number two, then the upper and lower bounds of Theorem~\ref{thm:join} coincide. For a triangle-free graph $G$, $\mob (G \vee K_1)$ could be either two or three. It is easily seen that for cycles we have $\mob (C_4 \vee K_1) = 2$ and $\mob (C_n \vee K_1) = 3$ for $n \geq 5$. The first example shows that the lower bound for $\mob (G \vee K_1)$ is tight, but it is an open question whether this can happen for graphs with large clique number. 
	
	\begin{corollary}
		If $G$ and $H$ both have clique number two, then $\mob (G \vee H) = 3$.
	\end{corollary}
	
	To show that the upper bound in Theorem~\ref{thm:join} is tight, consider the join $K_r^- \vee K_s ^-$, where $K_n^-$ represents a complete graph minus one edge. If $x_1,x_2$ is the pair of non-adjacent vertices in $K_r^-$ and $y_1,y_2$ is the pair of non-adjacent vertices of $K_s^-$, then $(V(K_r^-) \setminus \{ x_2\} ) \cup (V(K_s^-)\setminus \{ y_1,y_2\} )$ is a mobile general position set, as the set of occupied vertices forms a clique in $K_r^- \vee K_s^-$ and the robot at $x_1$ can follow the route $x_1 \move y_1 \move x_2 \move y_2$ to visit the remaining vertices. This matches the upper bound. More generally, the same argument works when $G$ and $H$ are both joins of cliques with empty graphs.
	
	To demonstrate sharpness of the lower bound, for $r \geq 2$ take the join $K_r \vee K^+_{r+1}$, where $K^+_{r+1}$ is the complete graph $K_{r+1}$ with an added leaf $x$. Suppose that a set of at least $r+2$ robots can traverse this graph in general position, and focus on the moment that there is a robot at $x$. As $r+2$ robots cannot be stationed on $K_{r+1}^+$, there must be a robot on $K_r$ and the positions occupied on $K_r$ and $K_{r+1}^+$ must both induce cliques. Hence every vertex of $K_r$ must contain a robot and in $K_{r+1}^+$ there is a robot at $x$ and its support vertex $x^{\prime }$. However, there are no legal moves in this configuration. Thus, $\mob(K_r \vee K^+_{r+1}) \leq r+1 = \min \{ \omega(K_r),\omega (K^+_{r+1})\} +1$.
	
	For an example of graphs with arbitrarily large clique number that meet the upper bound in $\mob (G \vee K_1) \leq \omega (G)+1$, consider the \emph{birdcage graph} $B_n$ formed as follows. Let $U$ be a clique on vertices $\{ u_1,\dots ,u_n\}$, $V$ be an empty graph on vertices $\{ v_1,\dots ,v_n\} $ and an additional vertex $z$, and add edges $u_i \sim v_i$ and $v_i \sim z$ for $i \in [n]$. The graph $B_n$ has clique number $n$ and we show that $n+1$ robots can traverse $B_n \vee K_1$ in general position. Denote the vertex of the $K_1$ by $x$. Start with robots at $U \cup \{x \} $ and make the move $x \move z$. Then for each $i \in [n]$ make the two moves $u_i \move v_i \move u_i$. It is easily seen that the robots are always in general position and visit all the vertices of $B_n \vee K_1$.

	%%%%%%%%%%%%%%%%%%%%%%%%%%%%
	\section{Concluding remarks and open problems}
	\label{sec:conclude}
	%%%%%%%%%%%%%%%%%%%%%%%%%%%%
	
	We conclude with a few open problems.
	
	\begin{itemize}
		\item Is there a non-trivial upper bound on $\mob(G\cp H)$, at least for the particular case $\mob(G\cp K_2)$?
		\item We have seen that $\mob(P_\infty\cp P_\infty)=4=\gp(P_\infty\cp P_\infty)$. In~\cite[Theorem 1]{KlavzarRus} it was proven that $\gp(P_\infty^{k,\Box})= 2^{2^{k-1}}$, where $P_\infty^{k,\Box}$ is the $k$-tuple Cartesian product of the infinite path $P_\infty$. It is therefore of interest to determine whether $\mob(P_\infty^{k,\Box})= 2^{2^{k-1}}$ holds for larger values of $k$.
		\item Are there graphs $G$ with arbitrarily large clique number such that $\mob (G \vee K_1) = 2$?
		\item In view of Proposition~\ref{prop:cylinders} and the preceding remarks, we ask what is the mobile general position number of cylinder graphs $C_7 \cp P_s$ and $C_{10} \cp P_s$ for $s \geq 5$?
		\item By Corollary~\ref{cor:Cr cp Cs}, $\mob (C_r \cp C_s) \geq 5$ if $r = 9$ or $r > 10$ and $s \geq 10$. It would be interesting to classify the mobile general position numbers of all torus graphs. 
		\item What is the mobile general position number of strong and direct products?
		\item What is the mobile general position number of the hypercube?

	\end{itemize}
	
	\section*{Acknowledgments}
	%%%%%%%%%%%%%%%%%%%%%%%%%%%%%%%%%%%%%%%%%%%%%%%%%%%%%%%
	
	Sandi Klav\v zar was supported by the Slovenian Research and Innovation Agency (ARIS) under the grants P1-0297, N1-0355, and N1-0285. D. Kuziak and I. G. Yero have been partially supported by the Spanish Ministry of Science and Innovation through the grant PID2023-146643NB-I00. D.\ Kuziak acknowledges the support from ``Ministerio de Educaci\'on, Cultura y Deporte'', Spain, under the ``Jos\'e Castillejo'' program for young researchers (reference number: CAS22/00081) to make a temporary visit to the University of Ljubljana, where this investigation has been partially developed. The research of James Tuite was partly carried out during a visit to University of Ljubljana supported by a grant from the Crowther Fund. Ethan Shallcross conducted this research with funding from an Open University research bursary. The authors thank the anonymous referees for their suggestions.
	
	\section*{Conflicts of interest}
	Sandi Klav\v{z}ar is an Associate Editor of the Discrete Applied Mathematics journal and was not involved in the review and decision-making process of this article. The authors do not have any additional conflicts of interest to declare.

\end{document}